\documentclass[12pt]{amsart}
\usepackage[a4paper, left=28mm, right=28mm, top=28mm, bottom=34mm]{geometry}
\allowdisplaybreaks

\usepackage{amssymb
,amsthm
,amsmath
,amscd
,mathtools	
,mathdots
,leftidx
,color
}

\usepackage[OT2,T1]{fontenc}
\usepackage[all]{xy}
\usepackage{url}
\usepackage[pdftex]{graphicx}

\usepackage[abbrev,alphabetic]{amsrefs}

\AtBeginDocument{%
  \def\MR#1{}
}

\DeclareMathOperator{\chara}{char}

\newcommand{\so}{\mathcal{O}}
\newcommand{\p}{\mathfrak{p}}

\newcommand\Z{\mathbb{Z}}

\newcommand\Q{\mathbb{Q}}

\newcommand\Pro{\mathbb{P}}

\newcommand\e{\textup{\'e}}
\newcommand\End{\textup{End}}
\newcommand\Aut{\textup{Aut}}

\newcommand\Gal{\mathrm{Gal}}
\newcommand\Shaf{\mathrm{Shaf}}

\newcommand\Pic{\mathrm{Pic}}

\newcommand\Spec{\mathop{\mathrm{Spec}}}

\newcommand\sh{\mathrm{sh}}

\newcommand\alb{\mathrm{alb}}
\newcommand\Alb{\mathrm{Alb}}
\newcommand\Image{\mathrm{Im}}
\newcommand\m{\mathfrak{m}}
\newcommand\sep{\mathrm{sep}}
\newcommand\id{\mathrm{id}}

\theoremstyle{plain}

\theoremstyle{definition}

\newtheorem*{lemma*}{Lemma}
\newtheorem*{prop*}{Proposition}
\newtheorem*{theorem*}{Theorem}
\newtheorem*{claim*}{Claim}
\newtheorem{definition*}{Definition}

\theoremstyle{plain}
\newtheorem{theoremsub}[subsection]{Theorem}
\newtheorem{propsub}[subsection]{Proposition}
\newtheorem{lemmasub}[subsection]{Lemma}

\theoremstyle{definition}

\newtheorem{definitionsub}[subsection]{Definition}
\newtheorem{remarksub}[subsection]{Remark}

\DeclareSymbolFont{cyrletters}{OT2}{wncyr}{m}{n}
\DeclareMathSymbol{\Sha}{\mathalpha}{cyrletters}{"58}

\usepackage{hyperref}
\begin{document}
\title{Reduction of bielliptic surfaces}
\author[T.\ Takamatsu]{Teppei Takamatsu}

%\date{\today}

\address{Department of Mathematics (Hakubi Center),
Kyoto University,
Kitashirakawa, Oiwake-cho, Sakyo-ku,
Kyoto 606-8502, JAPAN}
\email{teppeitakamatsu.math@gmail.com}


\begin{abstract}
A bielliptic surface (or hyperelliptic surface) is a smooth surface with a numerically trivial canonical divisor such that the Albanese morphism is an elliptic fibration.
In the first part of this paper, we study the structure of bielliptic surfaces over a field of characteristic different from $2$ and $3$, in order to prove the Shafarevich conjecture for bielliptic surfaces with rational points.
Furthermore, we demonstrate that the Shafarevich conjecture generally fails for bielliptic surfaces without rational points.
In particular, this paper completes the study of the Shafarevich conjecture for minimal surfaces of Kodaira dimension $0$.

In the second part of this paper, we study a N\'{e}ron model of a bielliptic surface. 
We establish the potential existence of a N\'{e}ron model for a bielliptic surface when the residual characteristic is not equal to $2$ or $3$.
\end{abstract}

\maketitle

\setcounter{section}{0}

\section{Introduction}

The Shafarevich conjecture for abelian varieties, proved by Faltings and Zarhin (see \cite[VI {$\S1$}, Theorem 2]{Faltings1992}) asserts the finiteness of isomorphism classes of abelian varieties of a fixed dimension over a fixed number field that admit good reduction away from a fixed finite set of finite places.
In \cite{Javanpeykar2017}, Javanpeykar and Loughran conjectured that the Shafarevich conjecture holds for more general families of varieties.
They also show that the Lang-Vojta conjecture for integral points of hyperbolic varieties implies the Shafarevich conjecture for hypersurfaces and complete intersections of general type (\cite[Theorem 1.5]{Javanpeykar2017}).
The Shafarevich conjecture is proved in many cases.
For example, this has been proved for del Pezzo surfaces (\cite{Scholl1985}), flag varieties (\cite{flag}), certain Fano threefolds (\cite{Javanpeykar2018}) (\cite{TakamatsuFano}), proper hyperbolic polycurves (\cite{Javanpeykar2015}, \cite{Nagamachi2019}), K3 surfaces (\cite{Andre1996}, \cite{She2017}, \cite{Takamatsu2020a}), Enriques surfaces (\cite{Takamatsu2020b}), hyper-k\"{a}hler varieties (\cite{Andre1996}, \cite{Takamatsuhk})).
Furthermore, it is verified for hypersurfaces in abelian varieties (\cite{Lawrence2020}) and very irregular varieties (\cite{Kramer}).
However, it is still open in general.

In the first part of this paper (Section \ref{sectionshaf}), we shall consider the same problem for bielliptic surfaces.
Our main theorem is the following.

\begin{theoremsub}[Theorem \ref{shaf}]\label{shafintro}
Let $F$ be a finitely generated field over $\Q$, and $R$ be a finite type algebra over $\Z$ which is a normal domain with fraction field $F$. Then, the set
\[
\left\{X \left|
\begin{array}{l}
X\colon \textup{bielliptic surface over }F \textup{ which admits a rational point},\\
X \textup{ has good reduction at any height }1 \textup{ prime ideal }\p \in \Spec R
\end{array}
\right.\right\}/F \textup{-isom}
\]
is finite.
\end{theoremsub}

\begin{remarksub}
\label{rem:newbiell}
In Theorem \ref{shafintro}, we only consider bielliptic surfaces with rational points.
This restriction is essential.
Indeed, in Proposition \ref{prop:withoutsection}, we show that the Shafarevich conjecture for general bielliptic surfaces fails.
Our counterexample is based on the example in (\cite[footnote 27]{Mazur}), which is a counterexample to the Shafarevich conjecture for genus 1 curves.
\end{remarksub}

As stated above, the Shafarevich conjecture is also justified in the cases of K3 surfaces and Enriques surfaces (see \cite{She2017}, \cite{Takamatsu2020a}, and \cite{Takamatsu2020b}).
Therefore, Theorem \ref{shafintro} completes the study of the Shafarevich conjecture for minimal surfaces of Kodaira dimension $0$.
We remark that we assume the existence of a rational point on $X$ in Theorem \ref{shafintro}, as in the case of abelian varieties. 

In Section \ref{sectionshaf}, first, we shall study the structure of bielliptic surfaces with a rational point and its reduction nature. 
After that, we give a proof of Theorem \ref{shafintro} by using that structure result and the Shafarevich conjecture for (products of) elliptic curves. 
Here, we also use the finitely generatedness of the Mordell-Weil groups of elliptic curves to reduce the problem to the case of products of elliptic curves.

In the second part (Section \ref{sectionneron}), we shall study a N\'{e}ron model of a bielliptic surface. 
Let $K$ be a discrete valuation field, and $X$ a smooth separated finite type scheme over $K$.
A N\'{e}ron model $\mathcal{X}$ of $X$ is a smooth separated finite type model over $\so_{K}$ satisfying some extension property, which is known as the N\'{e}ron mapping property.
N\'{e}ron proves the existence of such models for abelian varieties (\cite{Neron1964}). 
Moreover, Liu and Tong prove the existence of N\'{e}ron models for smooth proper curves of positive genus (see \cite{Liu2016} for a more precise statement).
In general, a N\'{e}ron model need not exist.
Our main theorem is the following.
\begin{theoremsub}[see Theorem \ref{gengenneron} for more precise statements.]\label{gengenneronintro}
Let $K$ be a strictly Henselian discrete valuation field with residue characteristic different from $2$ and $3$, and $X$ be a bielliptic surface over $K$.
Then $X$ potentially admits a N\'{e}ron model, i.e., there exists a finite separable extension $L/K$, such that $X_{L'}$ admits a N\'{e}ron model for any finite extension $L'/L$.
\end{theoremsub}

The key idea of Theorem \ref{gengenneronintro} is to take a quotient of a N\'{e}ron model of abelian surfaces. 
If $X$ satisfies some condition, we can show that this quotient is the N\'{e}ron model again (see Theorem\ref{genneron} for a sufficient condition).
However, in general, that quotient is not necessarily a N\'{e}ron model.
To treat the general case, we use some gluing arguments.

\subsection*{Notations and Terminologies}
\begin{itemize}
\item
Let $A \rightarrow B$ be a morphism of algebras.
For a scheme $X$ over $A$, we denote its base change $X\times_{A}B$ by $X_{B}$. 
\item
For any scheme $S$ and a point $s\in S$, we denote its residue field by $\kappa(s)$.
\item
For any discrete valuation field $K$, we denote its valuation ring by $\so_{K}$.
We denote the completion of $K$ by $\widehat{K}$.
\item
Let $K$ be a discrete valuation field and $X$ a smooth separated finite type scheme over $K$.
A scheme (resp.\ algebraic space) $\so_{K}$-model $(\mathcal{X}, i)$ of $X$ is a scheme (resp.\ algebraic space) $\mathcal{X}$ which is separated and of finite type over $\so_{K}$ with an isomorphism $i\colon \mathcal{X}_{K} \simeq X$.
We often omit $i$ and say that $\mathcal{X}$ is a scheme (resp.\ algebraic space) $\so_{K}$-model of $X$. 
We assume that "a $\so_K$-model" simply refers to a $\so_K$-scheme model.
An algebraic space $\so_{K}$-model $\mathcal{X}$ of $X$ is called smooth (resp.\ proper) (resp.\ projective) if  $\mathcal{X}$ is smooth (resp.\ proper) (resp.\ projective) over $\so_{K}$.
\item
Let $X$ be a smooth projective variety over $k$.
We denote the Albanese torsor of $X$ by $\Alb (X)$, and the Albanese morphism by
\[
\alb \colon X \rightarrow \Alb (X).
\]
\end{itemize}

\subsection*{Acknowledgments}
The author is deeply grateful to his advisor Naoki Imai for his deep encouragement and helpful advice. The author also would like to thank Yohsuke Matsuzawa and Tetsushi Ito for helpful suggestions.
The author also thanks the referee for constructive suggestions.
This work was supported by JSPS KAKENHI Grant number JP19J22795, JP22J00962, JP22KJ1780.

\section{Bielliptic surfaces and good reduction}\label{sectionshaf}
In this section, we first recall the definition of bielliptic surfaces and study their basic properties. 

\begin{definitionsub}
Let $k$ be a field.
Let $X$ be a smooth projective surface over $k$, and $Y$ a smooth projective curve over $k$.
Let $f \colon X \rightarrow Y$ be a proper surjective morphism over $k$.
We say $f$ is an \emph{elliptic fibration} if $f_{*} \so_{X} =\so_{Y}$ and the generic fiber of $f$ is a smooth genus 1 curve over the generic point of $Y$.
\end{definitionsub}

\begin{definitionsub}
Let $k$ be a field, and $X$ a smooth projective surface over $k$.
We say $X$ is a \emph{bielliptic surface over $k$} if
the Kodaira dimension $\kappa (X)$ is equal to $0$,  
the second Betti number $b_{2} (X)$ is equal to $2$, and 
the Albanese morphism $\alb \colon X \rightarrow \Alb (X)$ is an elliptic fibration.
\end{definitionsub}

Note that, we mainly treat the case where $\chara k \neq 2,3$ where the condition of the Albanese morphism is unnecessary (see \cite[p. 26]{BombieriMumford}).

First, we review the basic property of bielliptic surfaces over an algebraically closed field.

\begin{propsub}
\label{prop:fibration}
Let $k$ be a field of characteristic different from $2$ and $3$, and $X$ a bielliptic surface over $k$.
Then the Albanese torsor $\Alb (X)$ is 1-dimensional and
the Albanese morphism $\alb \colon X \rightarrow \Alb (X)$ is an elliptic fibration  
such that any geometric fiber is a smooth elliptic curve.
Moreover, $X_{\overline{k}}$ admits another elliptic fibration $g \colon X_{\overline{k}} \rightarrow \Pro^1_{\overline{k}}$, and all the fibers of $g$ are irreducible.
\end{propsub}

\begin{proof}
This follows from \cite[Theorem 8.6, Lemma 8.7 and Theorem 8.10]{Badescu2001}.
\end{proof}

It is well-known that a bielliptic surface over an algebraically closed field is written as a quotient of certain abelian surfaces (cf.\ \cite[Subsection 10.24]{Badescu2001}).
We shall generalize this fact to fields that are not necessarily algebraically closed.

\begin{lemmasub}\label{structure}
Let $k$ be a field of characteristic different from $2$ and $3$,
and $X$ a bielliptic surface over $k$ with a rational point $x \in X(k)$. 
Let $A' = \Alb (X)$ be the Albanese torsor of $X$, and $B$ the fiber $\alb^{-1}(\alb(x))$. 
By Proposition \ref{prop:fibration}, we regard $A'$ and $B$ as elliptic curves with rational points coming from $x$.
Then there exists an elliptic curve $A$ over $k$ which is isogenous to $A'$ and a finite \'{e}tale subgroup scheme $G \hookrightarrow A$ with a group scheme monomorphism $G \rightarrow \Aut_{B/k}$, such that there exists an isomorphism
\[
X \simeq_{k} (A \times B) /G
\]
which sends $x$ to $\overline{(0,0)}.$
Here, $\Aut_{B/k}$ is the automorphism scheme of the variety (rather than the group scheme) $B$.
Furthermore, the projection $X\rightarrow A/G$ via the above isomorphism is isomorphic to the Albanese morphism $\alb$, and $X \rightarrow B/G \simeq \Pro^{1}_{k}$ is an elliptic fibration.
\end{lemmasub}

\begin{proof}
We will prove this in three steps: first, when $k$ is an algebraically closed field; second, when $k$ is a separably closed field; and finally, in the case of a general field.

In the case when $k = \overline{k}$, then it follows from \cite[Subsection 10.24]{Badescu2001}.
For future convenience, we will outline the proof.
Let $g\colon X \rightarrow \Pro^1$ be an elliptic fibration given in Proposition \ref{prop:fibration}.
Let $S \subset \Pro^1$ be the image of the non-smooth locus of $g$, which is a closed subscheme of dimension $0$.
Then we can define the action 
\[
A' \times g^{-1}(\Pro^{1} \setminus S) \rightarrow g^{-1}(\Pro^{1} \setminus S)
\]
as follows.
For any $k$-algebra $R$, take $P \in g^{-1}(\Pro^{1}\setminus S) (R).$
Then we have the abelian scheme 
\[
T:= g^{-1} (g(P)) = g^{-1} (\Pro^{1} \setminus S) \times_{\Pro^{1} \setminus S} g(P)
\]
over $\Spec R$ with a section $P$, and we have a morphism $T \rightarrow A'_{R}$ coming from the composition 
\[
T \rightarrow g^{-1}(\Pro^{1}\setminus S) \hookrightarrow X \rightarrow A'.
\]
Then we have
\[
{A'}_{R} \simeq {A'}_{R}^{\vee} \rightarrow T^{\vee},
\]
and the right-hand side acts on $T$ canonically. Therefore we get the desired action.

By the minimality of $X$ (see \cite[Theorem 10.21]{Badescu2001}), one can extend the above action to $\sigma \colon A' \times X \rightarrow X$ (since $A'$ is an elliptic curve, it is enough to extend the action of the generic point of $A'$). 
Let $n$ be the intersection number 
of a fiber of $\alb$ and a fiber of $g$.

Then we have a diagram
\begin{equation}\label{xyalb}
\vcenter{
\xymatrix{
A' \times X \ar[r]^-{\sigma} \ar[d]^{\mathrm{id} \times \alb} & X \ar[d]^{\alb} \\
A' \times A' \ar[r]^-{\phi} &A' \\
}
}
\end{equation}

where $\phi (a, a'):= na + a'$. 
This diagram is commutative since it is commutative over $\overline{k}$ as in \cite[Subsection 10.24]{Badescu2001}.
By the above diagram, $B := \alb^{-1}(0)$ is stable under the action 
of $G' := A'[n]$ given by $\sigma$.
Therefore we have the morphism $G' \rightarrow \Aut_{B/k}$ corresponding to 
$\sigma \colon G' \times B \rightarrow B$. 
Moreover, we get the desired isomorphism
$(A' \times B) /G' \simeq X$  via the action $\sigma$ since it induces an isomorphism over $\overline{k}$. 
We put 
\begin{align*}
A &:= A'/ \ker(G' \rightarrow \Aut_{B/k}), \\
G &:= G'/ \ker(G' \rightarrow \Aut_{B/k})
\end{align*}
so that $G \rightarrow \Aut_{B/k}$ is injective.
These $A, B,$ and $G$ satisfy the desired conditions, and it finishes the proof for the case where $k= \overline{k}$.

Next, we treat the case where $k$ is separably closed. 
Then by the argument in the case where $k=\overline{k}$, we have a finite \'{e}tale morphism
$\pi \colon A'_{\overline{k}} \times B_{\overline{k}}\rightarrow X_{\overline{k}}$ which is the quotient map by an action of $G'_{\overline{k}}$, where $G':= A'[n]$ for some $n\in \{2,3,4,6,8,9\}$. 
We shall prove that this morphism is defined over $k$. 
The morphism $\pi$ can be decomposed as 
\[
\xymatrix
{
\pi\colon
A'_{\overline{k}} \times B_{\overline{k}} \ar[r]^-{\pi_{1}}_-{\simeq} &
\underline{\Spec}_{X_{\overline{k}}}\pi_{\ast}\so_{A'_{\overline{k}}\times B_{\overline{k}}} \ar[r]^-{\pi_{2}} &
X_{\overline{k}}.
}
\]
Here, $\pi_{\ast} \so_{A'_{\overline{k}}\times B_{\overline{k}}}$ is decomposed as $\oplus_{\chi \in \widehat{G'(\overline{k})}} \mathcal{F}_{\chi}$, where $\mathcal{F}_{\chi}$ is a line bundle on $X_{\overline{k}}$ associated with a character $\chi\in \widehat{G'(\overline{k})}$. 
Since the $n$-multiplication map on the Picard scheme $\Pic_{X/k}$ is an \'{e}tale morphism, each $\mathcal{F}_{\chi}$ descends to a line bundle on $X$, hence the morphism $\pi_{2}$ comes from the morphism 
\[
\pi_{2,0} \colon
Y:= \underline{\Spec}_{X}\oplus_{\chi \in \widehat{G'(\overline{k})}}\mathcal{F}_{\chi} \rightarrow X
\]
 over $k$. 
Since $\pi_{2,0}$ is a finite \'{e}tale morphism, the fiber over $x$ consists of $k$-rational points.
In particular, $\pi_{1}(0,0)$ descends to a $k$-rational point of $Y$. We can equip $Y$ with the structure of abelian variety over $k$ with the zero section $\pi_{1}(0,0)$ by the Albanese morphism.
Therefore the morphism $\pi_{1}$ is a homomorphism over $\overline{k}$ between abelian varieties defined over $k$. 
Since the homomorphism scheme of abelian varieties is \'{e}tale (\cite[Proposition 7.14]{Geer2014}), the morphism $\pi_{1}$ is also defined over $k$.
Hence our $\pi$ is defined over $k$, i.e., there exists a morphism 
$A' \times B \rightarrow X$.
Through this morphism, we have an isomorphism
$(A'\times B)/G'\rightarrow X$. 
This isomorphism induces the desired isomorphism
$
(A \times B)/G \rightarrow X,
$
where 
\begin{align*}
A &:= A'/\ker (G' \rightarrow \Aut_{B/k}), \\
G &:= G'/\ker (G' \rightarrow \Aut_{B/k}).
\end{align*}

Finally, we show the general case.
Note that, we have an elliptic fibration 
\[
g' \colon X_{k^{\sep}} \rightarrow B_{k^{\sep}}/G_{k^{\sep}} \simeq \Pro^{1}_{k^{\sep}}.
\]
Then by the argument in \cite[Proposition 5.6]{Creutz2018}, we obtain an elliptic fibration 
\[
g \colon X \rightarrow \Pro^{1}
\]
over $k$ such that $g_{k^{\sep}}$ is isomorphic to $g'$.
Using this $g$, we can prove the theorem in exactly the same way as in the case where $k = \overline{k}$.
It finishes the proof.
\end{proof}

\begin{remarksub}\label{structurealb}
In the proof of the above proposition, we take a quotient $A$ (resp.\ $G$) of $A'$ (resp.\ $G'$).
However, if we do not require $G \rightarrow \Aut_{B/k}$ to be injective, we do not have to take quotients in the final part of the proof.
More precisely, $A:=A'$ and $G:=G'$ gives the description
\[
X \simeq_{k} (A \times B)/G,
\]
where $A$ is the Albanese variety of $X$, $B := \alb^{-1}(\alb(x))$, $G:= A[n]$ for some $n \in \{2,3,4,6,8,9\}$ with a group scheme morphism $G \rightarrow \Aut_{B/k}$. 
In this description, the projection $X \rightarrow A/G \simeq A$ gives the Albanese morphism, where $A/G \simeq A$ is given by the $n$-multiplication map. Note that the integer $n$ is independent of the choice of a rational point $x\in X(K)$.
\end{remarksub}
\begin{definitionsub}\label{goodreduction}
Let $K$ be a discrete valuation field, and $X$ a bielliptic surface over $K$. 
We say $X$ admits \emph{good reduction} if there exists a smooth proper algebraic space $\mathcal{X}$ over $\so_{K}$ such that $\mathcal{X}_{K}\simeq X$.
\end{definitionsub}

Here, we also recall the definition of a N\'{e}ron model.

\begin{definitionsub}\label{sakiNeron}
Let $K$ be a discrete valuation field, $X$ a smooth separated finite type scheme over $K$, and $\mathcal{X}$ a smooth $\so_{K}$-model of $X$.
The $\so_K$-model $\mathcal{X}$ is a N\'{e}ron model of $X$ if for any smooth $\so_{K}$-scheme $Z$ and $K$-morphism $u_{K} \colon Z_{K} \rightarrow X$, there exists a unique extension $u \colon Z \rightarrow \mathcal{X}$ of $u_{K}$.
\end{definitionsub}

\begin{remarksub}\label{remneron}
In Definition \ref{sakiNeron}, if an extension $u$ exists, then it is automatically unique since $\mathcal{X}$ is separated over $\so_{K}$.
\end{remarksub}

\begin{lemmasub}\label{unique}
Let $K$ be a discrete valuation field with residue field $k$ whose characteristic is different from $2$ and $3$. 
For any bielliptic surface $X$ over $K$, the following are equivalent.
\begin{enumerate}
\item
$X$ admits good reduction.
\item
There exists a smooth proper scheme $\so_{K}$-model $\mathcal{X}$ of  $X$.
\item
There exists a smooth projective scheme $\so_{K}$-model $\mathcal{X}$ of $X$.
\end{enumerate}
\end{lemmasub}

\begin{proof}
$(3) \Rightarrow (2) \Rightarrow (1)$ is clear. We shall show $(1)\Rightarrow (3)$. 
Let $\mathcal{X}$ be a smooth proper algebraic space $\so_{K}$-model of $X$. 
Let $\widehat{\mathcal{X}}$ be a formal completion of $\mathcal{X}_{\widehat{\so_{K}}}$ along its special fiber $\mathcal{X}_{k}$, where $\widehat{\so_{K}}$ is a completion of $\so_{K}$. 
We note that $\mathcal{X}_{k}$ is a bielliptic surface over $k$. 
Indeed, since $\mathcal{X}_{k}$ is a principal Cartier divisor on $\mathcal{X}$, we have $\omega^{\otimes m}_{\mathcal{X}/\so_{K}} \simeq \so_{\mathcal{X}}$, where $m$ is the order of $\omega_{X/K}$. 
Therefore $\omega_{\mathcal{X}_{k}/k}^{\otimes m}$ is trivial, and $\mathcal{X}_{k}$ is a minimal surface of Kodaira dimension $0$ with the same Betti numbers as $X$. 
Hence $\mathcal{X}_{k}$ is a bielliptic surface, and we have $H^{2}(\mathcal{X}_{k}, \so_{\mathcal{X}_{k}})=0$. 
By \cite[Corollary 8.5.5]{Fantechi2005}, one can lift an ample line bundle $L_{k}$ on $\mathcal{X}_{k}$ to $\widehat{L}$ on $\widehat{\mathcal{X}}$.
By Grothendieck's existence theorem (\cite[Theorem 6.3]{Knutson1971}), the line bundle $\widehat{L}$ also lifts to $L$ on $\mathcal{X}_{\widehat{\so_{K}}}$ as an invertible sheaf. 
The categorical equivalence of Grothendieck's existence theorem implies that for any coherent sheaf $F$ on $\mathcal{X}_{\widehat{\so_{K}}}$, $F\otimes L^{\otimes n}$ is globally generated for sufficiently large $n$. 
This also holds over a finite scheme cover of $\mathcal{X}_{\widehat{\so_{K}}}$, hence $L$ is an ample line bundle.
Therefore, we have a smooth projective scheme model $\mathcal{X}_{\widehat{\so_{K}}}$ over $\widehat{\so_{K}}$.
Let $\phi \colon \omega_{\mathcal{X}_{\widehat{\so_{K}}}/\widehat{\so_{K}}}^{\otimes m}\simeq \so_{\mathcal{X}_{\widehat{\so_{K}}}}$ be the base change of the isomorphism $\omega^{\otimes m}_{\mathcal{X}/\so_{K}} \simeq \so_{\mathcal{X}}$.
We take a cyclic covering $\mathcal{Y}$ associated with $\phi$, i.e.,
\[
Y := \underline{\mathop{\mathrm{Spec}}}_{\mathcal{X}_{\widehat{\so_{K}}}} \bigoplus_{r=0}^{m-1} \omega_{\mathcal{X}_{\widehat{\so_{K}}/\so_{K}}}^{\otimes r},
\]
where the module on the right-hand side is equipped with the algebra structure given by $\phi$.
Since $m$ is $2, 3, 4,$ or $6$ (see \cite[p.37]{BombieriMumford}), $\mathcal{Y}$ is finite \'{e}tale over $\mathcal{X}_{\widehat{\so_{K}}}$.
The base change $Y_{\overline{\widehat{K}}}$ of the generic fiber  
$Y := \mathcal{Y}_{\widehat{K}}$ is an abelian surface since its canonical sheaf is trivial and its first Betti number is not equal to zero.
Moreover, by Hensel's lemma, there exists a finite unramified extension $L/ \widehat{K}$ such that $Y_{L}$ admits an $L$-rational point, i.e., $Y_{L}$ has a structure of an abelian variety. 
Therefore, by \cite[Proposition 1.4.2]{Bosch1990}, the $\so_L$-model $\mathcal{Y}_{\so_{L}}$ satisfies the N\'{e}ron mapping property, and $\mathcal{Y}$ also satisfies the N\'{e}ron mapping property by a descent argument.
Using \cite[Theorem 7.2.1]{Bosch1990}, the smooth proper N\'{e}ron model $\mathcal{Y}$ descends to the smooth proper N\'{e}ron model $\mathcal{Y}_{0}$ over $\so_{K}$ whose generic fiber is $Y_{0}$ which is the cyclic covering of $X$ associated with $\phi_{0}$. 
By the N\'{e}ron mapping property, the cyclic action on $Y_{0}$ can be extended to that on $\mathcal{Y}_{0}$, and this action is free since its base change to $\widehat{\so_{K}}$ is the cyclic action on $\mathcal{Y}$.
Therefore, we can take a finite $\e$tale quotient $\mathcal{Y}_{0}/\mu_{m}$ which is separated of finite type. Hence it gives a smooth proper scheme model over $\so_{K}$ of $X$.
Note that $\mathcal{Y}_0$ is projective over $\so_{K}$ by \cite[Theorem 6.4.1]{Bosch1990}, hence $\mathcal{Y}_{0}/\mu_{m}$ is also a smooth projective.
\end{proof}

\begin{remarksub}\label{modelunique}
In general, if a smooth proper variety $X$ over a discrete valuation field admits a smooth proper scheme $\so_{K}$-model and a N\'{e}ron model, then the smooth proper scheme model of $X$ is isomorphic to the N\'{e}ron model, by van der Waerden's purity theorem (see \cite[Corollaire (21.12.16)]{Grothendieck1967}).
Hence a smooth proper $\so_{K}$-model of $X$ is unique for such $X$.
Note that, for general $X$, a smooth proper scheme $\so_K$-model of $X$ need not be unique (see \cite[Remark 6.3]{LiedtkeMatsumoto}).
\end{remarksub}

\begin{lemmasub}\label{neronsub}
Let $K$ be a discrete valuation field with residue field $k$ whose characteristic is $p$. Let $A$ be an abelian variety over $K$. Let $G \subset A$ be a finite subgroup scheme over $K$. 
Suppose that the order of $G(\overline{K})$ is coprime to $p$.
Let $\mathcal{A}$ be the N\'{e}ron model of $A$.
Then there exists the N\'{e}ron model $\mathcal{G}$ of $G$
and a natural closed immersion $\mathcal{G} \rightarrow \mathcal{A}$. 
\end{lemmasub}
\begin{proof}
Let $\mathcal{A}'$ be the N\'{e}ron model of $A/G$.
Then the natural morphism $\pi: \mathcal{A} \rightarrow \mathcal{A}'$ is an \'{e}tale isogeny.
Indeed, by \cite[Proposition 7.3.6]{Bosch1990}, we have a morphism $\pi' \colon \mathcal{A}' \rightarrow \mathcal{A}$ with $\pi' \circ \pi = n$, where $n$ is the order of $G(\overline{K})$.
Since $n$ is coprime to $p$, $\pi' \circ \pi$ is \'{e}tale (\cite[Lemma 7.3.2]{Bosch1990}). Hence, $\pi$ is quasi-finite and flat (\cite[Proposition 2.4.2, Lemma 7.3.1]{Bosch1990}). 
Moreover, the kernel of $\pi$ is \'{e}tale, since $\pi' \circ \pi$ is \'{e}tale. So $\pi$ is \'{e}tale.
Let $\mathcal{K}$ be the kernel of $\pi$, which is \'{e}tale separated of finite type over $\so_{K}$.
Then $\mathcal{G}$ is canonically isomorphic to $\mathcal{K}$, since for any smooth $\so_{K}$-scheme $Z$,
\begin{align*}
\mathcal{K}(Z) &= \ker (\mathcal{A}(Z)\rightarrow \mathcal{A}'(Z))\\ 
&= \ker (A(Z_{K})\rightarrow (A/G)(Z_{K})) \\
&= G(Z_{K}).
\qedhere
\end{align*}
\end{proof}

\begin{propsub}\label{criteria}
Let $K$ be a discrete valuation field with residue field $k$ whose characteristic is different from $2$ and $3$. Let $X$ be a bielliptic surface over $K$ which admits a rational point $x\in X(K)$. Let $X \simeq (A \times B) /G$ be the isomorphism given in Lemma \ref{structure}. Then the following are equivalent.
\begin{enumerate}
\item
The bielliptic surface $X$ admits good reduction.
\item
The elliptic curves $A$ and $B$ admit good reduction.
\end{enumerate}
\end{propsub}

\begin{proof}
First we shall show $(1)\Rightarrow (2)$. 
By Lemma \ref{unique}, we have $\mathcal{X} \rightarrow S:= \Spec \so_{K}$ which is a smooth projective scheme $\so_{K}$-model of $X$. 
Note that $\Pic_{\mathcal{X}_{\kappa(s)}/\kappa(s)}$ is smooth of dimension $1$ for any geometric point $s$ on $S$ since $\mathcal{X}_{\kappa(s)}/\kappa(s)$ is a bielliptic surface over a field of characteristic different from $2$ and $3$. 
Let $\Pic_{\mathcal{X}/S}$ be the Picard functor, and $\mathcal{Y} := \Pic^{0}_{\mathcal{X}/S} \hookrightarrow \Pic_{\mathcal{X}/S}$ is the open closed subgroup scheme whose fiber over any geometric point $s$ of $S$ is equal to the identity component of $\Pic_{\mathcal{X}_{s}/\kappa(s)}$ (see \cite[Theorem 9.4.8, Proposition 9.5.20]{Fantechi2005}). 
Therefore, $\mathcal{Y}$ is an abelian scheme over $S$. For the fixed rational point $x\in \mathcal{X}(\so_{K})$, we have the corresponding morphism $\alb \colon \mathcal{X} \rightarrow \Pic^{0}_{\mathcal{Y}/S}$ which sends $x$ to $0$. 
Note that this morphism is flat since one can check it fiberwise (see \cite[Proposition 2.4.2]{Bosch1990}).
Since $\Pic^{0}_{\mathcal{Y}/S}$ is a smooth proper scheme $\so_{K}$-model of the Albanese variety of $X$, the elliptic curve $A$ admits good reduction. 
Moreover, one can show that $\mathcal{Z}:= \mathcal{X}\times_{\Pic^{0}_{\mathcal{Y}/S}}S$ is a smooth proper scheme $\so_{K}$-model of $B$, where $S \rightarrow \Pic^{0}_{\mathcal{Y}/S}$ is the $0$-section. Indeed, by the definition $\mathcal{Z}$ is proper and flat over $S$, and whose geometric fibers are smooth. Therefore $B$ admits good reduction too.

Next, we shall show $(2)\Rightarrow (1)$. Let $\mathcal{A}, \mathcal{B}$ be the N\'{e}ron models of $A$, $B$ (which is the unique smooth proper scheme $\so_{K}$-model of the elliptic curves $A, B$). Moreover, let $\mathcal{G}$ be the N\'{e}ron model of $G$. 
Since $A$ admits good reduction, the scheme $G$ consists of spectra of fields that are unramified over $K$ (recall that $G$ is a quotient of the $n$-torsion points of $A'$, where $A'$ is an elliptic curve which is isogenous to $A$ and $n \in \{ 2,3,4,6,8,9 \}$). 
Therefore a N\'{e}ron model of $G$ is a finite \'{e}tale group scheme over $\so_{K}$. 
By the N\'{e}ron mapping property, one can extend the action of $G$ on $A \times B$ to the action of $\mathcal{G}$ on $\mathcal{A} \times \mathcal{B}$. 
Be Lemma \ref{neronsub}, $\mathcal{G} \hookrightarrow \mathcal{A}$ is a closed immersion of group schemes.
Therefore the above action $\mathcal{G} \times \mathcal{A} \times \mathcal{B} \rightarrow \mathcal{A} \times \mathcal{B}$ is a free action by the locally free group scheme. 
Hence one can take a finite \'{e}tale quotient $\mathcal{A} \times \mathcal{B} \rightarrow \mathcal{X}$, where $\mathcal{X}$ is separated of finite type over $\so_{K}$. 
Therefore one can show that $\mathcal{X}$ is a smooth proper scheme $\so_{K}$-model of $X$, and it finishes the proof.
\end{proof}

\begin{remarksub}
This proposition is one formulation of a good reduction criterion for bielliptic surfaces. By using a $\mu_{m}$-torsor of a bielliptic surface $X$ defined as a cyclic covering defined by $\omega_X$, one can formulate a similar statement for bielliptic surfaces without rational points (see \cite[Section 9]{Chiarellotto2016}).
\end{remarksub}

\begin{theoremsub}\label{shaf}
Let $F$ be a finitely generated field over $\Q$, and $R$ a finite type algebra over $\Z$ which is a normal domain with fraction field $F$. Then, the set
\[
\Shaf :=
\left\{X \left|
\begin{array}{l}
X\colon \textup{bielliptic surface over }F \textup{ which admits a rational point},\\
X \textup{ has good reduction at any height }1 \textup{ prime ideal }\p \in \Spec R
\end{array}
\right.\right\}/F \textup{-isom}
\]
is finite.
\end{theoremsub}

\begin{proof}
Shrinking $\Spec R$, we may assume that $R$ is smooth over $\Z$.
For $X \in \Shaf$, one can associate a rational point $x\in X (F)$ and we have a description $X \simeq (A \times B)/G$ as in Lemma \ref{structure}. 
By Proposition \ref{criteria}, the elliptic curves $A$ and $B$ admit good reduction at any height $1$ prime $\p \in \Spec R$. Therefore there exists a map $\Shaf \rightarrow \Shaf_{\textup{ab}}$, where $\Shaf_{\textup{ab}}$ is the set of pairs of $F$-isomorphism classes of elliptic curves both of which admit good reduction at any height $1$ prime of $\Spec R$. 
By \cite[VI, \S1, Theorem 2]{Faltings1992}, $\Shaf_{\textup{ab}}$ is a finite set. 
Therefore, by Lemma \ref{lem:finitefiber} below, we obtain the desired finiteness.
\end{proof}

\begin{lemmasub}
\label{lem:finitefiber}
Let $F$ be a finitely generated field over $\Q$, and $R$ a smooth domain over $\Z$ with fraction field $F$.
We put $\Shaf$ as in Theorem \ref{shaf}.
Also, we put $\Shaf_{\mathrm{ab}}$ as in the proof of Theorem \ref{shaf}, i.e.
\[
\Shaf_{\mathrm{ab}}
:= 
\left\{
(A,B)\left| 
\begin{array}{l}
A, B \colon \textup{elliptic curves over } F
\textup{ that have good reduction}\\
\textup{ at any height } 1 \textup{ prime ideal } \p \in \Spec R
\end{array}
\right\}
\right.
/F\textup{-isom}.
\]
Let
\[
\Shaf \rightarrow \Shaf_{\mathrm{ab}}, \quad X \mapsto (A,B)
\]
be the morphism defined by choosing a rational point $x \in X(F)$ and using Lemma \ref{structure} and Proposition \ref{criteria} (see the proof of Theorem \ref{shaf}).
Then this morphism is finite-to-one.
\end{lemmasub}

\begin{proof}
For fixed $A \times B \in \Shaf_{\textup{ab}}$, let $X\in \Shaf$ be a bielliptic surface lying in its fiber.
Then, $X$ is isomorphic to $(A \times B) /G_{X}$ for some finite \'{e}tale subgroup scheme
$G_{X} \hookrightarrow A$ and some embedding $\alpha_{X} \colon G_{X} \hookrightarrow \Aut_{B/F}$.
Since the order of $G_{X}$ is bounded, the candidates of $G_{X}\hookrightarrow A$ are finitely many.
Hence we may fix an embedding $\iota \colon G := G_{X} \hookrightarrow A$.
Let $S_{A, B, \iota}$ be the set of embeddings $\alpha \colon G \hookrightarrow \Aut_{B/F}$ such that $(A\times B) / G$ is a bielliptic surface, where the quotient is taken with respect to the action $\iota \times \alpha$.
Let $\sim$ be the equivalence relation on $S_{A,B, \iota}$ given by $B(F)$-conjugate.
More precisely, for $\alpha, \alpha' \in \Aut_{B/F}(G)$, $\alpha \sim \alpha'$ if and only if there exists $b \in B(F)$ such that $t_{b} \alpha t_{-b} = \alpha'$, where $t_{b} \colon B \rightarrow B$ is the translation which sends $0$ to $b$. 
In this case, we have the $F$-isomorphism
\[
((A \times B)/G)^{(\iota \times \alpha)} \simeq ((A \times B)/G)^{(\iota \times \alpha')}
\]
induced by $\id \times t_{b}$, where the left-hand (resp.\ right-hand) side is the quotient with respect to the action $\iota \times \alpha$ (resp.\ $\iota \times \alpha'$). 
Therefore, it suffices to show the finiteness of $S_{A,B,\iota}/\sim$.
Since $(B/G)_{\overline{F}}^{(\alpha_{\overline{F}})}$ (the quotient with respect to the action $\alpha_{\overline{F}}$) is $\mathbb{P}^{1}_{\overline{F}}$, the image $H$ of $\alpha_{\overline{F}} (G(\overline{F}))$ in $\Aut(B_{\overline{F}})/B(\overline{F}) \simeq \Aut(B_{\overline{F}},0)$ is non-trivial. Indeed, if this is trivial, then the quotient is given by an isogenous elliptic curve, and we have a contradiction.
Take an element $g = g_{\alpha} \in G(\overline{F})$ such that the image of $\alpha_{\overline{F}} (g)$ generates $H$. 
Let $b_{\alpha} \in B(\overline{F})$ be a fixed point of $\alpha_{\overline{F}} (g)$. 
Since  $\alpha_{\overline{F}}(G(\overline{F}))$ is commutative, we have a decomposition $G(\overline{F}) \simeq H_{1}^{\alpha} \times H_{2}^{\alpha}$, with
$\alpha (H_{1}^{\alpha}) \subset \Aut (B_{\overline{F}},b_{\alpha})$ 
and $\alpha(H_{2}^{\alpha}) \subset B(\overline{F}) \subset \Aut (B_{\overline{F}})$ (see the arguments in \cite[Subsection 10.26, List 10.27]{Badescu2001}).
We denote the projection $G(\overline{F}) \rightarrow H_{i}^{\alpha}$ by $p_{i}^{\alpha}$.
Since $\Aut (B_{\overline{F}}, b_{\alpha})$ is $\Z/2\Z$, $\Z/4\Z$, or $\Z/6\Z$, there are classification of $H_{1}^{\alpha} \hookrightarrow \Aut (B_{\overline{F}}, b_{\alpha})$ and $H_{2}^{\alpha} \hookrightarrow B(\overline{F})$ (see \cite[List 10.27]{Badescu2001}). In particular, we remark that if $H_{2}^{\alpha}$ is non-trivial then 
$H_{2}^\alpha$ is $\Z/2\Z$ (the case of $(a_{2})$ or $(c_{2})$ in \cite[List 10.27]{Badescu2001}) or $\Z/3\Z$ (the case of $(c_{1})$ in \cite[List 10.27]{Badescu2001}).
In both cases, we note that $\# H_{2}^{\alpha}$ divides $\# H_{1}^{\alpha}$.

We can define the morphism of sets 
\[
\phi \colon S_{A,B,\iota} \rightarrow B(F) ; \alpha \mapsto \sum_{g \in G(\overline{F})}\alpha_{\overline{F}}(g) (0).
\]
Since $\# H_{2}^{\alpha} | \# H_{1}^{\alpha}$, we have
\[
\sum_{g\in G(\overline{F})} \alpha_{\overline{F}}(p_{2}^{\alpha}(g))(0) = 
\sum_{h \in H_{2}^{\alpha}} \# H_{1}^{\alpha} \cdot \alpha_{\overline{F}}(h)(0) =0.
\]
Let $n$ be the order of $G(\overline{F})$, and 
$h^{\alpha,g} \in \Aut(B_{\overline{F}},0)$ the automorphism such that
\[
\alpha_{\overline{F}}(p_{1}^{\alpha}(g)) = t_{b_{\alpha}} h^{\alpha,g} t_{b_{\alpha}}^{-1}.
\]
Since
\[
h^{\alpha,g_{0}}(\sum_{g\in G(\overline{F})} h^{\alpha,g}) =\sum_{g\in G(\overline{F})} h^{\alpha,g}
\]
 for non-trivial $h^{\alpha,g_{0}}$ (one can take $g_{0} = g_{\alpha}$),
 the image $(\sum_{g\in G(\overline{F})} h^{\alpha,g}) (B_{\overline{F}})$ is contained in the fixed locus of  $h^{\alpha, g_{0}}$.
Since $(\sum_{g\in G(\overline{F})} h^{\alpha,g}) (B_{\overline{F}})$ is connected and containing $0$, we have $\sum_{g\in G(\overline{F})} h^{\alpha,g} = 0 \in \End (B_{\overline{F}})$.
Combining with
\[
\alpha_{\overline{F}}(g)(0) = \alpha_{\overline{F}}(p_{1}^{\alpha}(g))(0) + \alpha_{\overline{F}}(p_{2}^{\alpha}(g))(0),
\] 
we have 
\begin{align*}
\phi (\alpha) &= \sum_{g\in G(\overline{F})} \alpha_{\overline{F}}(p_{1}^{\alpha}(g))(0) \\
&= \sum_{g\in G(\overline{F})} ( h^{\alpha,g} (-b_{\alpha}) + b_{\alpha})\\
&= \sum_{g\in G(\overline{F})} (1 - h^{\alpha,g}) (b_{\alpha})\\
&= n b_{\alpha}.
\end{align*}
Since we fix the isomorphism class of $G$, 
possible isomorphism classes $H_{1}^{\alpha}, H_{2}^{\alpha}$ are at most finitely many (in fact unique).
Therefore, possible $\alpha |_{H_{2}^{\alpha}}$ are at most finitely many.
Moreover, if $n b_{\alpha}$ is fixed, then possible $b_{\alpha}$ are finitely many, and $\alpha |_{H_{1}^{\alpha}}$ are at most finitely many since possible embeddings $H_{1}^{\alpha} \rightarrow \Aut(B_{\overline{F}}, b_{\alpha})$ are finitely many.
Hence each fiber of $\phi$ is finite.
We note that $\phi (t_{b} \alpha t_{b}^{-1}) = n(b_{\alpha}+b)$.
Therefore $\phi$ induces the morphism with finite fibers
\[
S_{A,B,\iota}/ \sim \rightarrow B(F)/ n B(F).
\]
By \cite{Neron1952} (see also \cite[Section 1]{Lang1959}), the abelian group $B(F)$ is finitely generated, so the right-hand side is finite. It finishes the proof.
\end{proof}

In Theorem \ref{shaf}, we only consider bielliptic surfaces admitting rational points.
We construct an example below to show that this assumption is an essential one.
This example is based on a counterexample to the Shafarevich conjecture for genus 1 curves in \cite[footnote 27]{Mazur}.

\begin{propsub}
\label{prop:withoutsection}
There exists a finite set of finite prime numbers $S$ such that the set 
\[
\Shaf' :=
\left\{X \left|
\begin{array}{l}
X\colon \textup{bielliptic surface over }\Q,\\
X \textup{ has good reduction at any} \textup{ prime number }p \notin S
\end{array}
\right.\right\}/F \textup{-isom}
\]
is an infinite set.
\end{propsub}

\begin{proof}
We take a finite set of finite prime numbers $S$ and an elliptic curve $E$ over $\Q$ satisfying the following:
\begin{enumerate}
\item
$2 \notin S,$
\item
$E$ has good reduction outside $S$,
\item
The Mordell-Weil rank of $E$ is 0, and the analytic rank of $E$ is 0.
\item
Any point of $E[2]$ is a $\Q$-rational point.
\end{enumerate}
Note that the assumption $(2)$ ensures that $S$ is non-empty by \cite{Ogg}.
Moreover, by Kolyvagin's result \cite{Kolyvagin1988} (see also \cite[Theorem 1]{Kolyvagin1991}) and the assumption (3),
the Tate-Shafarevich group $\Sha (E)$ of $E$ is known to be finite.
We can take, for example, $S$ as $\{3,5 \}$ and $E$ as the elliptic curve over $\Q$ with Cremona label $15a2$.
We shall show that $\Shaf'$ is an infinite set for this $S$.

By Tate's argument (\cite[footnote 27]{Mazur}) and the assumptions (2) and (3), we can take infinitely many isomorphism classes of $E$-torsors $C_{i}$ over $\Q$ ($i=1,2, \ldots$) such that $C_{i, \Q_{p}}$ is a trivial torsor for any $p \notin S$.
Since $\Sha(E)$ is finite, we may assume that there exists a prime number $p \in S$ such that 
the isomorphism classes of $E_{\Q_{p}}$-torsor $C_{i, \Q_{p}}$ are all distinct.
By restricting the action of the torsor structure, we have a free action $\sigma_{1}$ of $E[2]$ on $C_{i}$.
We fix a $\Z/2\Z$-basis $P, Q$ of $E[2] \simeq (\Z/2\Z)^{\oplus 2}$, that are $\Q$-rational points by the assumption (4).
We consider the morphism
\[
\sigma_{2} \colon E[2] \times E \rightarrow E; (aP+bQ, R) \mapsto (-1)^a R + bQ,
\]
which is well-defined and gives the action of $E[2]$ on $E$.
We put
\[
X_{i} := (C_{i} \times E)/ E[2],
\]
where the quotient is taken with respect to the action $\sigma := \sigma_{1} \times \sigma_{2}$.
Clearly, $X_{i}$ is a bielliptic surface over $\Q$.
Moreover, $X_{i}$ has good reduction outside $S$ by the same argument as in the proof of Proposition \ref{criteria}.
We shall show that $X_{i}$ $(i=1,2, \ldots)$ represent infinitely many $\Q$-isomorphism classes.
Let $A_{i}$ be the Albanese torsor $\Alb (X_{i})$ of $X_{i}$.
It suffices to show that $A_{i}$ $(i=1,2, \ldots)$ represent infinitely many $\Q$-isomorphism classes.
Let  $f_{i} \colon X_{i} \rightarrow C_{i}/ E[2]$  be the natural morphism.
Fix a $\overline{\Q}$-rational point of $C_{i, \overline{\Q}}$ so that $C_{i, \overline{\Q}}$ is equipped with the structure of an elliptic curve, which is isomorphic to $E_{\overline{\Q}}$.
Since 
\[
(E_{\overline{\Q}} \times E_{\overline{\Q}})/E[2]  \simeq 
 X_{i, \overline{\Q}} = (C_{i, \overline{\Q}} \times E_{\overline{\Q}}) /E[2] 
\rightarrow
  C_{i, \overline{\Q}} / E[2] \simeq E_{\overline{\Q}} / E[2] \rightarrow E_{\overline{\Q}}
 \] 
 is an Albanese morphism by the proof of Lemma \ref{structure}, we have
 $A_{i} \simeq C_{i}/E[2]$.
Therefore, $A_{i}$ admits $E/E[2] \simeq E$-torsor structure, whose class in Weil–Ch\^{a}telet group $H^1(\Gal(\overline{\Q}/\Q), E)$ is $2[C_{i}]$.
Here, $C_{i}$ is a class represented by the $E$-torsor $C_{i}$.
We recall that $[C_{i, \Q_{p}}] \in H^1 (\Gal (\overline{\Q_{p}},\Q_{p}), E_{\Q_{p}})$ is a non-trivial class.
Moreover, since $p \neq 2$ by the assumption (1),
the group $H^1 (\Gal (\overline{\Q_{p}},\Q_{p}), E_{\Q_{p}}) [2] \simeq E(\Q_{p})/ 2 E(\Q_{p})$ is a finite group.
Since the classes $[C_{i,\Q_{p}}] \in H^1 (\Gal (\overline{\Q_{p}},\Q_{p}), E_{\Q_{p}})$
are all distinct, $\{2[C_{i,\Q_{p}}] \} \subset H^1 (\Gal (\overline{\Q_{p}},\Q_{p}), E_{\Q_{p}})$ is an infinite set, and it finishes the proof.
\end{proof}

\section{On the N\'{e}ron model of a bielliptic surface}\label{sectionneron}
In this section, we prove the existence of a N\'{e}ron model of a bielliptic surface under certain conditions. First, we recall the notion of a weak N\'{e}ron model.
\begin{definitionsub}\label{Neron}
Let $K$ be a discrete valuation field, $X_{K}$ be a smooth separated finite type scheme over $K$, and $\mathcal{X}$ and $\mathcal{X}_{i}$ $(i=0, \ldots, s)$ smooth $\so_{K}$-models of $X$. %(i.e.\ satisfying $\mathcal{X}_{K} \simeq X$, $(\mathcal{X}_{i})_{K} \simeq X$).
The family of $\so_K$-models $\{\mathcal{X}_{i}\}_i$ is a \emph{weak N\'{e}ron model} of $X$ if each $K^{\sh}$-valued point of $X$ extends to an $\so_{K^{\sh}}$-valued point of at least one of $\mathcal{X}_{i}$.
We say the $\so_K$-model $\mathcal{X}$ is a \emph{weak N\'{e}ron model} of $X$ if $\{\mathcal{X}\}$ is a weak N\'{e}ron model of $X$ in the sense defined above.
\end{definitionsub}

A weak N\'{e}ron model satisfies the following useful extension property.

\begin{propsub}\label{weakneronproperty}
Let $K$ be a discrete valuation field, and $X$ be a smooth separated $K$-scheme of finite type.
\begin{enumerate}
\item
Let $\mathcal{X}$ be a weak N\'{e}ron model of $X$. Then for any smooth $\so_{K}$-scheme $Z$ and for any $K$-rational map $u_{K} \colon Z_{K} \dashrightarrow X$, there exists an extension of $u_{K}$ to an $\so_{K}$-rational map $Z \dashrightarrow \mathcal{X}$, i.e., there exists an open subscheme $U \subset Z$ and an $\so_{K}$-morphism $u \colon U \rightarrow \mathcal{X}$ which is an extension of $u_{K}$ such that $U_{\kappa(s)}$ is open dense in $Z_{\kappa(s)}$ for any $s \in \Spec \so_{K}$.
\item
Let $\{\mathcal{X}_{i}\}_{0\leq i \leq s}$ be a weak N\'{e}ron model of $X$. Then for any smooth $\so_{K}$-scheme $Z$ with irreducible special fiber and for any $K$-rational map $Z_{K} \dashrightarrow X$, there exists an integer $i$ with $0 \leq i \leq s$ such that $u_{K}$ extends to an $\so_{K}$-rational map $Z \dashrightarrow \mathcal{X}_{i}$.
\end{enumerate}
\end{propsub}
\begin{proof}
See \cite[Proposition 3.5.3]{Bosch1990}.
\end{proof}

\begin{lemmasub}\label{finet}
Let $K$ be a discrete valuation field, $X$ a smooth separated $K$-scheme of finite type, and $\mathcal{X}$ the N\'{e}ron model of $X$.
Let $\mathcal{Y}$ be a smooth separated finite type scheme over $\so_{K}$.
Let $f \colon \mathcal{X} \rightarrow \mathcal{Y}$ be a finite \'{e}tale morphism 
over $\so_{K}$. Assume that $\mathcal{Y}$ is a weak N\'{e}ron model of its generic fiber $\mathcal{Y}_{K}$. 
Then $\mathcal{Y}$ is the N\'{e}ron model of $\mathcal{Y}_{K}$.
\end{lemmasub}
\begin{proof}
Let $Z$ be a smooth $\so_{K}$-scheme, and $u_{K} \colon Z_{K} \rightarrow \mathcal{Y}_{K}$ a $K$-morphism. 
We shall extend the domain of $u_{K}$ to $Z$.
By Proposition \ref{weakneronproperty}, there exists an open subscheme $U \subset Z$ and $u \colon U \rightarrow \mathcal{Y}$, such that $U$ contains the generic fiber $Z_{K}$ and all the generic points of the special fiber, and $u$ is the extension of $u_{K}$. Using $u$, one has a finite \'{e}tale covering $\widetilde{U}:= U \times_{\mathcal{Y}}\mathcal{X} \rightarrow U$. 
Since $Z$ is a regular scheme, by taking the normalization of $Z$ in $\widetilde{U}$ and using the Zariski-Nagata purity theorem, one can extend this finite \'{e}tale covering to $\widetilde{Z} \rightarrow Z$. 
By the N\'{e}ron mapping property of $\mathcal{X}$, we have $\mathcal{X}(\widetilde{U})=\mathcal{X}(\widetilde{U}_{K})= \mathcal{X}(\widetilde{Z})$. 
Therefore, the base change morphism $\widetilde{u} \colon \widetilde{U} \rightarrow \mathcal{X}$ can be extended to $\widetilde{u}\colon \widetilde{Z} \rightarrow \mathcal{X}$. 
Then we have a morphism $u_{\widetilde{Z}} \colon \widetilde{Z} \rightarrow \mathcal{Y}$. 
Consider the equalizer diagram
\begin{equation}
\vcenter{
\xymatrix{
\mathcal{Y}(Z) \ar[r] \ar[d]
& \mathcal{Y}(\widetilde{Z}) \ar@<0.5ex>[r] \ar@<-0.5ex>[r] \ar[d]
& \mathcal{Y}(\widetilde{Z}\times_{Z}\widetilde{Z}) \ar[d] \\
\mathcal{Y}(Z_{K}) \ar[r]
&\mathcal{Y}(\widetilde{Z}_{K}) \ar@<0.5ex>[r] \ar@<-0.5ex>[r]
&\mathcal{Y}((\widetilde{Z}\times_{Z}\widetilde{Z})_{K})
}
},
\end{equation}
where the vertical arrows are injective since $Z$, $\widetilde{Z}$, and $\widetilde{Z} \times_{Z}\widetilde{Z}$ are reduced schemes which are flat over $\so_{K}$ and $\mathcal{Y}$ is separated over $\so_{K}$.
It suffices to show that $u_{\widetilde{Z}} \in \mathcal{Y}(\widetilde{Z})$ is contained in the equalizer of the upper right arrows.
Since the middle vertical map sends $u_{\widetilde{Z}}$ to $u_{K}$, which is contained in the equalizer of the lower right arrows, it finishes the proof.
\end{proof}

\begin{remarksub}
Without the assumption that $\mathcal{Y}$ is a weak N\'{e}ron model, Lemma \ref{finet} does not hold.
For example, consider an elliptic curve $E$ over a discrete valuation field $K$ whose $\ell$-torsion points are $K$-rational. 
Here, we fix a prime number $\ell$ that is different from the residual characteristic of $K$.
Assume that the residue field of $K$ is algebraically closed, and the special fiber of minimal regular model of $E$ is of reduction type $I_n$ (i.e. non-singular rational curves arranged in the sphe of an $n$-gon) with $\ell | n$.
Let $\mathcal{E}$ and $\mathcal{G}$ be N\'{e}ron models of $E$ and $G:= E[\ell]$ with a closed immersion $\mathcal{G} \hookrightarrow \mathcal{E}$ given by Lemma \ref{neronsub}.
Then $\mathcal{E} \rightarrow \mathcal{E}/\mathcal{G}$ is a finite \'{e}tale morphism over $\so_K$, but $\mathcal{E}/\mathcal{G}$ is not a N\'{e}ron model of $E/G$.
Suppose by contradiction that $\mathcal{E}/\mathcal{G}$ is a N\'{e}ron model of $E/G$.
Then the isomorphism $E/G \simeq E$ induced by $\times \ell$ extends to an isomorphism
$\mathcal{E}/\mathcal{G} \simeq \mathcal{E}$.
Since the composition $E \rightarrow E/G \simeq E$ is the $\ell$-multiplication, so is the composition
$\mathcal{E} \rightarrow \mathcal{E}/\mathcal{G} \simeq \mathcal{E}$.
Since the later composition is not surjective by $\ell|n$, we obtain the contradiction.
\end{remarksub}

We obtain the existence of a N\'{e}ron model for biellptic surfaces admitting good reduction.
Note that this can also be proved using \cite[Proposition 6.2]{GLL}.

\begin{propsub}\label{neron}
Let $K$ be a discrete valuation field with residual characteristic different from $2$ and $3$, and $X$ a bielliptic surface over $K$. Assume that $X$ admits good reduction. Then $X$ admits a N\'{e}ron model.
\end{propsub}

\begin{proof}
Let $\mathcal{X}$ be a proper smooth scheme in Lemma \ref{unique}. Then the cyclic covering over $\so_{K}$ associated with a fixed isomorphism $\omega_{\mathcal{X}/\so_{K}}^{\otimes m} \simeq \so_{\mathcal{X}}$ is the N\'{e}ron model of its generic fiber as in the proof in Lemma \ref{unique}. 
By the valuative criterion of properness, the model $\mathcal{X}$ is a weak N\'{e}ron model of $X$, so by Lemma \ref{finet}, it finishes the proof.
\end{proof}

Next, we state the existence of N\'{e}ron models in a little more general setting.
Let $K$ be a discrete valuation field with residue characteristic different from $2$ and $3$, and $X$ a bielliptic surface over $K$ which admits a rational point $x\in X(K)$. Let $X \simeq (A\times B) /G$ be the isomorphism given in Lemma \ref{structure}.
Let $\mathcal{A}$, $\mathcal{B}$ and $\mathcal{G}$ be N\'{e}ron models of $A$, $B$ and $G$.
Let $\overline{\mathcal{A}}$, $\overline{\mathcal{B}}$ be minimal regular models of $A$, $B$. Note that the N\'{e}ron model of an elliptic curve is the smooth locus of the minimal regular model (see \cite[Proposition 1.5.1]{Bosch1990}).
By the N\'{e}ron mapping property, we have a group action $\sigma\colon \mathcal{G} \times \mathcal{A} \times \mathcal{B} \rightarrow \mathcal{A} \times \mathcal{B}$, which is a free action since $\mathcal{G} \hookrightarrow \mathcal{A}$ is a closed immersion 
by Lemma \ref{neronsub}.
In the following, we assume that $\mathcal{G}$ is finite \'{e}tale over $\so_{K}$ (i.e., $G$ admits good reduction).
Then one can extend the action $\mathcal{G} \times \mathcal{A} \rightarrow \mathcal{A}$ to $\mathcal{G} \times \overline{\mathcal{A}} \rightarrow \overline{\mathcal{A}}$ uniquely. 
Indeed, the connected component of $\mathcal{G}$ is the spectrum of discrete valuation ring $\so_{L}$ which is unramified over $\so_{K}$.
Since $\mathcal{A}_{\so_{L}}$ is also a minimal regular model of $A_{L}$ (\cite[Proposition 9.3.28]{Liu2002}), one can use the minimality to extend this action (\cite[Proposition 9.3.13]{Liu2002}).
Similarly, one can extend the action $\mathcal{G} \times \mathcal{B} \rightarrow \mathcal{B}$ uniquely, therefore we have the action $\Sigma \colon \mathcal{G}\times \overline{\mathcal{A}} \times \overline{\mathcal{B}} \rightarrow \overline{\mathcal{A}}\times \overline{\mathcal{B}}$.

\begin{theoremsub}\label{genneron}
Let $A, B, G, \mathcal{A}, \mathcal{B}, \mathcal{G} ,\sigma$ and $\Sigma$ as above (so we suppose that $G$ admits good reduction).
Then, if $\Sigma$ is a free action (for example, the case where $A$ admits good reduction), the quotient $(\mathcal{A}\times \mathcal{B}) /\mathcal{G}$ is the N\'{e}ron model of $X$. 
\end{theoremsub}

\begin{proof}
By the assumption, the scheme $(\overline{\mathcal{A}}\times \overline{\mathcal{B}})/\mathcal{G}$ is a finite \'{e}tale quotient of $\overline{\mathcal{A}} \times \overline{\mathcal{B}}$. 
Therefore $(\overline{\mathcal{A}} \times \overline{\mathcal{B}})/\mathcal{G}$ is a regular scheme which is proper over $\so_{K}$. 
Since its smooth locus is $(\mathcal{A} \times \mathcal{B})/\mathcal{G}$, by the valuative criterion of the properness and 
\cite[Lemma 3.1]{Liu2016}, one can show that $(\mathcal{A}\times \mathcal{B})/\mathcal{G}$ is a weak N\'{e}ron model of $X$. By Lemma \ref{finet}, it finishes the proof.
\end{proof}

This proof is based on the philosophy `the smooth locus of a minimal model is the N\'{e}ron model'. 
On the other hand, one can prove the existence of N\'{e}ron models in a more general setting, as follows.  

\begin{theoremsub}\label{gengenneron}
Let $K$ be a strictly Henselian discrete valuation field with residue field $k$ whose characteristic is different from $2$ and $3$, and $X$ a bielliptic surface over $K$ which admits a rational point $x \in X(K)$. Let $X \simeq (A \times B) /G$ be the isomorphism given in Remark \ref{structurealb}.
If $G$ admits good reduction (i.e., the inertia group $I_{K}$ acts trivially on $G(\overline{K})$),
then $X$ admits a N\'{e}ron model. In particular, for any bielliptic surface $X$ over $K$, there exists a finite separable extension $L/K$ such that $X_{L'}$ admits a N\'{e}ron model for any finite extension $L'/L$.
\end{theoremsub}
\begin{remarksub}
Since we work over a strictly Henselian discrete valuation field, 
if $X(K) = \emptyset$, then $X \rightarrow \Spec K \rightarrow \Spec \so_{K}$ itself is the N\'{e}ron model of $X$ by Hensel's lemma.
Thus, to prove the existence of a N\'{e}ron model, we may assume that $X(K) \neq \emptyset$.
\end{remarksub}
\begin{proof}
It suffices to show the first statement.
Let $X \simeq (A \times B)/G$ be a bielliptic surface with a rational point $x \in X(K)$, as in the assumption.
Let $n$ be the integer as in Remark \ref{structurealb}. 
Since we are working with a strictly Henselian discrete valuation field, the following hold.
\begin{itemize}
\item
The group scheme $G$ consists of $K$-rational points.
\item
Every connected component of the special fiber of a N\'{e}ron model $\mathcal{A}$ has a $k$-rational point. Especially, every component is geometrically connected.
\item
Let 
\[
\Phi\colon
X(K) \rightarrow A/G(K) \rightarrow A(K) \simeq \mathcal{A}(\so_{K}) \rightarrow \pi_{0}(\mathcal{A}_{k})/n\pi_{0}(\mathcal{A}_{k})
\]
be the composition of the Albanese morphism with the reduction maps.
Here, the second arrow is an isomorphism given by the $n$-multiplication map as in Remark \ref{structurealb}.
Then, there exist $K$-valued points $x_{0}=x, x_{1}, \ldots, x_{s} \in X(K)$ which give a complete set of representatives of $\Image \Phi$. 
\end{itemize}
For each rational points $x_{i}$, we have a description 
\[
X \simeq (A_{i} \times B_{i}) /G_{i}
\]
given in Remark \ref{structurealb}. 
Here, by the definition every $A_{i}$ is the Albanese variety of $X$, so every $A_{i}$ is naturally isomorphic to each other. The only difference is that the zero section of $A_{i}$ is given by $\alb(x_{i})$. As in Remark \ref{structurealb}, each $G_{i}$ consists of $K$-rational points.

Let $\mathcal{A}_{i}, \mathcal{B}_{i}, \mathcal{G}_{i}$ be the N\'{e}ron models of $A_{i}, B_{i}, G_{i}$. By the N\'{e}ron mapping property, one can extend the group action to N\'{e}ron models. 
By Lemma \ref{neronsub}, a N\'{e}ron model $\mathcal{G}_{i}$ is a finite \'{e}tale subgroup scheme of $\mathcal{A}_{i}$.
Let $\mathcal{X}_{i}$ be the finite \'{e}tale quotient $(\mathcal{A}_{i} \times \mathcal{B}_{i})/\mathcal{G}_{i}$. As before, the quotient $\mathcal{X}_{i}$ is smooth separated of finite type over $\so_{K}$.
Let $\mathcal{X}$ be the scheme obtained by gluing $\mathcal{X}_{i}$ together on the generic fibers.
By definition, $\mathcal{X}$ is a smooth finite type scheme over $\so_{K}$ satisfying $\mathcal{X}_{K}\simeq X$. We shall prove that this $\mathcal{X}$ is the desired N\'{e}ron model. \\
First, we shall prove that $\mathcal{X}$ is separated over $\mathcal{O}_K$. 
Let $R$ be a discrete valuation ring over $\so_{K}$,
and $F$ the fraction field of $R$. 
By the valuative criterion of separatedness, it is enough to show that $\mathcal{X}(R) \rightarrow \mathcal{X}(F)$ is injective. 
Assume that there exist $t_{1}, t_{2} \colon \Spec R \rightarrow \mathcal{X}$ which are different $R$-valued points going to the same $F$-valued point. 
Clearly, each $t_{l}$ factors through some $\mathcal{X}_{i_{l}} \hookrightarrow \mathcal{X}$.
If $t_{l}$ factors through the same component $\mathcal{X}_{i}$, by the separatedness of $\mathcal{X}_{i}$, we have $t_{1}=t_{2}$. 
Therefore each $t_{l}$ factors through $\mathcal{X}_{i_{l}}$ with $i_{1}\neq i_{2}$ and each $t_{l}$ does not factors through the generic fiber $X$. 
Let 
$s_{l} := \alb(t_{l}) \in \mathcal{A}_{i_{l}} (R),$ 
where $\alb$ is the extension of the Albanese morphism to $\mathcal{X} \rightarrow \mathcal{A}$. 
The morphism $\alb|_{\mathcal{X}_{i_{l}}}$ is written as the following composition,
\[ 
\mathcal{X}_{i_{l}} := (\mathcal{A}_{i_{l}} \times \mathcal{B}_{i_{l}}) /\mathcal{G}_{i_{l}} \rightarrow \mathcal{A}_{i_{l}}/\mathcal{G}_{i_{l}} \rightarrow \mathcal{A}_{i_{l}},
\]
where the last map is given by $n$-multiplication morphism. 
As remarked above, for any $0\leq i,j \leq s$, the varieties $\mathcal{A}_{i}$ and $\mathcal{A}_{j}$ are naturally isomorphic (we denote it by $\mathcal{A}$), and there exists the following commutative diagram.
\begin{equation}
\vcenter{
\xymatrix{
\mathcal{A} \ar[r]^{t_{i,j}} \ar[d]^{n_{\mathcal{A}_{i}}} & \mathcal{A} \ar[d]^{n_{\mathcal{A}_{j}}} \\
\mathcal{A} \ar[r]^{t_{i,j}} & \mathcal{A}
}
}
\end{equation}
Here $t_{i,j}$ is a translation on $\mathcal{A}$ which sends $\alb(x_{i})$ to $\alb(x_{j})$, and $n_{\mathcal{A}_{i}}$ is $n$-multiplication map as the group scheme $\mathcal{A}_{i}$.  Note that we identify $\alb(x_{i})$ with its extension to the $\so_{K}$-valued point of $\mathcal{A}$.
Let $\m$ be the maximal ideal of $\Spec R$. Since $t_{l} (\m)$ is contained in the special fiber of $\mathcal{X}_{i_{l}}$, the point $s_{l} (\m)$ is contained in the special fiber of $\mathcal{A}$. 
By the above description of $\alb|_{\mathcal{X}_{i_{l}}}$, 
the point $s_{l}(\m)$ lies in 
the component which is contained in $\Pi_{i_{l}}$, where 
\[
\Pi_{i} := \Image (n_{\mathcal{A}_{i}} \colon \pi_{0}(\mathcal{A}_{k}) \rightarrow \pi_{0}(\mathcal{A}_{k})).
\] 
By the above commutative diagram, as subsets of the component group $\pi_{0}(\mathcal{A}_{k})$, we have 
\[
\Pi_{i_{l}} = t_{0,i_{l}}\Pi_{0}= (\alb(x_{i_{l}})-\alb{(x_{0}}))\Pi_{0}.
\]
Since we have $i_{1}\neq i_{2}$, by the definition of $x_{i}$, we have 
that $\Pi_{i_{1}}$ and $\Pi_{i_{2}}$ are disjoint in $\pi_{0}(\mathcal{A}_{k}).$
So we have $s_{1} \neq s_{2}$. 
However, since $s_{1}$ and $s_{2}$ give the same $F$-valued point of $\mathcal{A}$, we have $s_{1}=s_{2}$ by the separatedness of $\mathcal{A}$. This is a contradiction, thus $\mathcal{X}$ is separated over $\so_{K}$.

Next, we shall prove that the family $\{\mathcal{X}_{i}\}_{i=0, \ldots, s}$ is a weak N\'{e}ron model of $X$. Let $y \in X(K)$ be any $K$-valued point of $X$.
By definition, there exists a unique $x_{i} \in X(K)$ which represents $\Phi (y)$. 
Let $\alb(y)' \in \mathcal{A}(\so_{K})$ be the unique extension of $\alb(y)$.
Then the special fiber of $\alb (y)'$ lies in the union of components corresponding to $t_{0,i} \Pi_{0}=\Pi_{i}$.
We note that this union of components is the image of the $n$-multiplication map on the special fiber induced by $n_{\mathcal{A}_{i}}$ (see \cite[Lemma 7.3.1, Lemma 7.3.2]{Bosch1990}). 
Since $\so_{K}$ is strictly Henselian, the fiber of $\alb(y)'$ along $n_{\mathcal{A}_{i}}$ has a component $\Spec \so_{K}$, i.e., there exists a lift $y'_{1} \in \mathcal{A}(\so_{K})$ of $\alb(y)'$ along $n_{\mathcal{A}_{i}}$. 
Let $y_{1}\in A(K)$ be the $K$-rational point of $A$ induced by $y'_{1}$ (so $y_{1}$ is a lift of $\alb(y)$ along the $n$-multiplication map on $A_{i}$).  
Let $y_{2} := (-y_{1}) \cdot y \in X(K)$, where $\cdot$ denotes the action $A_{i}\times X \rightarrow X$ given in the proof of Lemma \ref{structure}. 
Then by the definition 
(see the commutative diagram (\ref{xyalb}) in the proof of Lemma \ref{structure}), 
we have $\alb(y_{2}) =\alb (x_{i})$, and therefore $y_{2}$ lies in $B_{i}(K)$ and we have $y=\overline{(y_{1}, y_{2})}$. 
We can extend $(y_{1}, y_{2}) \in A_{i} \times B_{i}$ to $(y'_{1}, y'_{2}) \in \mathcal{A}_{i}\times \mathcal{B}_{i}(\so_{K})$. 
Therefore $y' := \overline{(y'_{1}, y'_{2})} \in \mathcal{X}_{i}(\so_{K})$ gives a desired extension of $y$. 

Finally, we shall prove that $\mathcal{X}$ satisfies the N\'{e}ron mapping property. This part is essentially the same as in the proof of Lemma \ref{finet}, but we include it for the sake of completeness.
Let $Z$ be any smooth $\so_{K}$-scheme, and $u_{K} \colon Z_{K}\rightarrow X$ be a $K$-morphism. 
By Remark \ref{remneron}, it suffices to show the existence of an extension of $u$.
Moreover, we may assume that $Z$ has an irreducible special fiber (in the general case, one can glue them). 
By Proposition \ref{weakneronproperty} (2), there exists an extension of $u_{K}$ to an $\so_{K}$-rational map $u\colon Z \dashrightarrow \mathcal{X}_{i}$ for some $0 \leq i \leq s$. Therefore, there exists an open subscheme $U \hookrightarrow Z$ where $u$ is defined, such that $U$ contains $Z_{K}$ and the generic point of $Z_{k}$.
Then the finite \'{e}tale covering 
\[
\widetilde{U}:= U\times_{\mathcal{X}_{i}} (\mathcal{A}_{i} \times \mathcal{B}_{i}) \rightarrow U
\]
can be extended to the finite \'{e}tale covering $\widetilde{Z} \rightarrow Z$ by taking the normalization of $Z$ in $\widetilde{U}$ and using the Zariski-Nagata purity theorem. By the N\'{e}ron mapping property, we have 
\[
\mathcal{A}_{i} \times \mathcal{B}_{i}(\widetilde{U}) = \mathcal{A}_{i} \times \mathcal{B}_{i} (\widetilde{U}_{K}) = \mathcal{A}_{i} \times \mathcal{B}_{i}(\widetilde{Z}),
\]
so the base change morphism $\widetilde{u} \colon \widetilde{U} \rightarrow \mathcal{A}_{i} \times \mathcal{B}_{i}$ can be extended to $\widetilde{u} \colon \widetilde{Z}\rightarrow \mathcal{A}_{i} \times \mathcal{B}_{i}$. 
Hence we have a morphism $u_{\widetilde{Z}} \colon \widetilde{Z} \rightarrow \mathcal{X}_{i}$, and by the uniqueness of extension, one can descends $u_{\widetilde{Z}}$ to $u_{Z}$ which is a desired extension of $u_{K}$. It finishes the proof.
\end{proof}


\end{document}